\documentclass[11pt]{amsart}

\usepackage{amsfonts}
\usepackage{amsthm}
\usepackage{amsmath, amssymb}
\usepackage{thmtools, thm-restate}
\usepackage{enumerate}
\usepackage{graphicx}
\usepackage{color}
\usepackage[percent]{overpic}

\usepackage[colorlinks,linkcolor=blue,citecolor=blue]{hyperref}

\usepackage[top=2.5cm, bottom=2.5cm, left=2.5cm, right=2.5cm]{geometry}

\newtheorem{thm}{Theorem}
\newtheorem{prop}[thm]{Proposition}
\newtheorem{corol}[thm]{Corollary}
\newtheorem{lem}[thm]{Lemma}

\newcommand{\gtop}{g_4^{\rm top}}

\newcommand{\ualg}{u_{\rm alg}}
\newcommand{\qfloor}{\left\lfloor\frac{q}{2}\right\rfloor}
\newcommand{\lk}{{\rm lk}\,}

\newcommand{\spin}{\ifmmode{\rm Spin}\else{${\rm spin}$\ }\fi}
\newcommand{\spinc}{\ifmmode{{\rm Spin}^c}\else{${\rm spin}^c$}\fi}

\begin{document}

\title{Gaps between consecutive untwisting numbers.}

\author{Duncan McCoy}
\address{Department of Mathematics \\
         The University of Texas At Austin}

\begin{abstract}
For $p\geq 1$ one can define a generalization of the unknotting number $tu_p$ called the $p$th untwisting number which counts the number of null-homologous twists on at most $2p$ strands required to convert the knot to the unknot. We show that for any $p\geq 2$ the difference between the consecutive untwisting numbers $tu_{p-1}$ and $tu_p$ can be arbitrarily large. We also show that torus knots exhibit arbitrarily large gaps between $tu_1$ and $tu_2$.
\end{abstract}

\maketitle

\section{Introduction}
Given a knot $K$ in $S^3$, we perform a {\em null-homologous twist} by taking an unknotted curve $C$ disjoint from $K$ with $\lk(C,K)=0$ and performing $+1$-surgery or $-1$-surgery on $C$. If $C$ bounds an embedded disk intersecting $K$ transversely in $2p$ points, then we call this a {\em null-homologous twist on $2p$ strands}. Such a twist can always be performed locally by adding a full twist on $2p$ parallel strands with appropriate orientations. An example of a null-homologous twist on four strands is shown in Figure~\ref{fig:sample_operation}. 

\begin{figure}
  \begin{overpic}[width=0.3\textwidth]{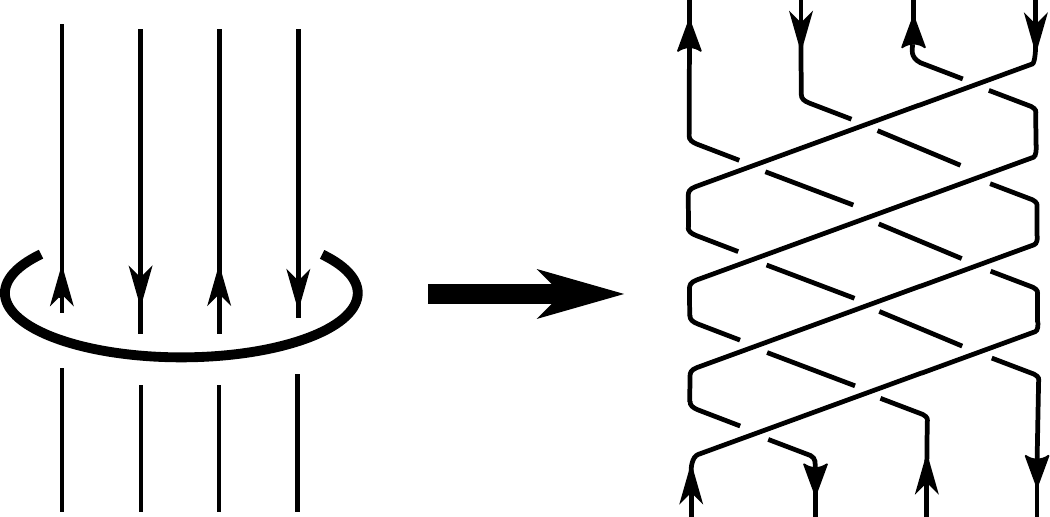}
    \put (-11,18) {$-1$}
  \end{overpic}
  \caption{A null-homologous twist on $4$ strands.}
  \label{fig:sample_operation}
  \end{figure}

Ince used null-homologous twisting operations to define an infinite sequence of generalizations to the unknotting number \cite{Ince16Untwisting}. For a knot $K$ the {\em $p$th untwisting number}, denoted $tu_p(K)$, is the minimum number of null-homologous twists on at most $2p$ strands required to convert $K$ to the unknot. Since a null-homologous twist on two strands is equivalent to a standard crossing change, $tu_1$ coincides with the classical unknotting number. 
One may also define the {\em untwisting number} by $tu(K)=\min tu_p(K)$. 
Clearly the untwisting numbers form a decreasing sequence:
\[u(K)=tu_1(K)\geq tu_2(K) \geq \dots \geq tu_{p-1}(K) \geq tu_{p}(K)\geq \dots \geq tu(K).\]
The main purpose of this article is to show that the difference between consecutive pairs of untwisting numbers can be arbitrarily large.

\begin{restatable}{thm}{unboundedgaps}\label{thm:gaps}
For any pair of positive integers $p \geq 2$ and $m\geq 1$, there is a knot $K$ such that
\[
tu_{p-1}(K)-tu_{p}(K) \geq m.
\]
\end{restatable}
The gaps between untwisting numbers have previously been studied by Ince, who showed that the gap between $tu_1$ and $tu_2$ can be arbitrarily large \cite{Ince16Untwisting}. Ince also considered the separation between higher untwisting numbers, showing for example that for any $p\geq 1$ the gap between $tu_p$ and $tu$ can be arbitrarily large (cf. \cite[Example~6.5]{Ince17UntwistingHF}). Our examples are similar to those studied by Ince, however we are able to establish stronger results through better lower bounds on $tu_p$. These lower bounds are provided by relating $tu_p$ and the smooth slice genus $g_4(K)$:
\begin{equation}\label{eq:lowerbound}
tu_p(K)\geq \frac{g_4(K)}{p}.
\end{equation}
Here $g_4(K)$ denotes smooth slice genus of $K$. For fixed $p$, the lower bound in \eqref{eq:lowerbound} turns out to be optimal as the knots used to prove Theorem~\ref{thm:gaps} will be knots attaining equality in \eqref{eq:lowerbound}.

Whilst \eqref{eq:lowerbound} shows that the $tu_p$ admit lower bounds based on the smooth slice genus, these lower bounds do not yield any information about $tu$. It turns out that one can obtain lower bounds on $tu$ using the topological slice genus:
\begin{equation}\label{eq:tu_lowerbound}
tu(K)\geq \gtop(K).
\end{equation}
This can be seen from results of Ince \cite{Ince16Untwisting}, who used the work of Borodzik and Friedl \cite{Borodzik14Algebraic, Borodzik15UnknottingI} to show that $tu(K)\geq \ualg(K)$. Alternatively one can establish \eqref{eq:tu_lowerbound} using the concept of algebraic genus \cite{McCoy19algebraic}.

Given that the unknotting numbers of torus knots were notoriously hard to compute, it is natural to wonder what one can say about the behaviour of untwisting numbers for torus knots. For torus knots with braid index at least four the untwisting number, $tu_2$ is strictly smaller than the unknotting number.

\begin{restatable}{thm}{untwistingtorus}\label{thm:untwisting_torus}
If $\min\{p,q\}\geq 4$, then $tu_{2}(T_{p,q})<u(T_{p,q})$. Furthermore, for any $p,q>1$ we have
\begin{equation}\label{eq:tu_2_upperbound}
tu_2(T_{p,q})\leq \frac{3}{8}pq.
\end{equation}
\end{restatable}
Since the unknotting number satisfies $u(T_{p,q})=\frac12(p-1)(q-1)$, it follows that for torus knots the difference between $tu_1$ and $tu_2$ grows arbitrarily large as the braid index increases. For torus knots with braid index two, i.e those of the form $K=T_{2,p}$ we have $tu(K)=u(K)=\frac12 |\sigma(K)|$. This follows from the fact that the classical knot signature provides a lower bound for  $\gtop(K)$ and hence for $tu(K)$. The same reasoning shows that $tu(T_{3,4})=u(T_{3,4})=3$ and $tu(T_{3,5})=u(T_{3,5})=4$. However, for the  remaining torus knots of braid index three understanding their untwisting numbers seems much more challenging.

\section{Unbounded gaps}
First we prove the following proposition, which implies \eqref{eq:lowerbound}. 
\begin{prop}\label{prop:smooth_twisting}
If $K$ and $K'$ are knots related related by a null-homologous twist on $2p$ strands, then
\[
|g_4(K)-g_4(K')|\leq p.
\]
\end{prop}
\begin{proof} We observe that a null-homologous twist on $2p$ strands can be accomplished by $2p$ oriented band moves. This can be proven by induction on $p$. Consider a full twist on $2p$ strands with $p$ strands oriented up and $p$ strands oriented down. Such a twist can be arranged as a full twist on $2p-2$ strands with two more strands, one oriented up and the other down, ``wrapping around'' the full twist as in the left hand side of Figure~\ref{fig:band_additions}. As illustrated in Figure~\ref{fig:band_additions} one can perform two oriented band moves and isotopies to produce a full twist on $2p-2$ strands with two parallel strands alongside. Thus, proceeding inductively, we see that the full twist on $2p$ strands can be converted to $2p$ parallel strands by $2p$ oriented band moves.

Thus if $K$ and $K'$ are related by a null-homologous twist on $2p$ strands, then there is a sequence of $2p$ oriented band moves and isotopies that  convert $K$ into $K'$. These moves allow one to construct a smoothly embedded surface $F$ of genus $p$ properly embedded in $S^3 \times [0,1]$ so that $\partial F = K \times\{0\} \cup K'\times \{1\}$. Thus
\[
|g_4(K)-g_4(K')|\leq p,
\]
as required.
\begin{figure}
  \begin{overpic}[width=0.95\textwidth]{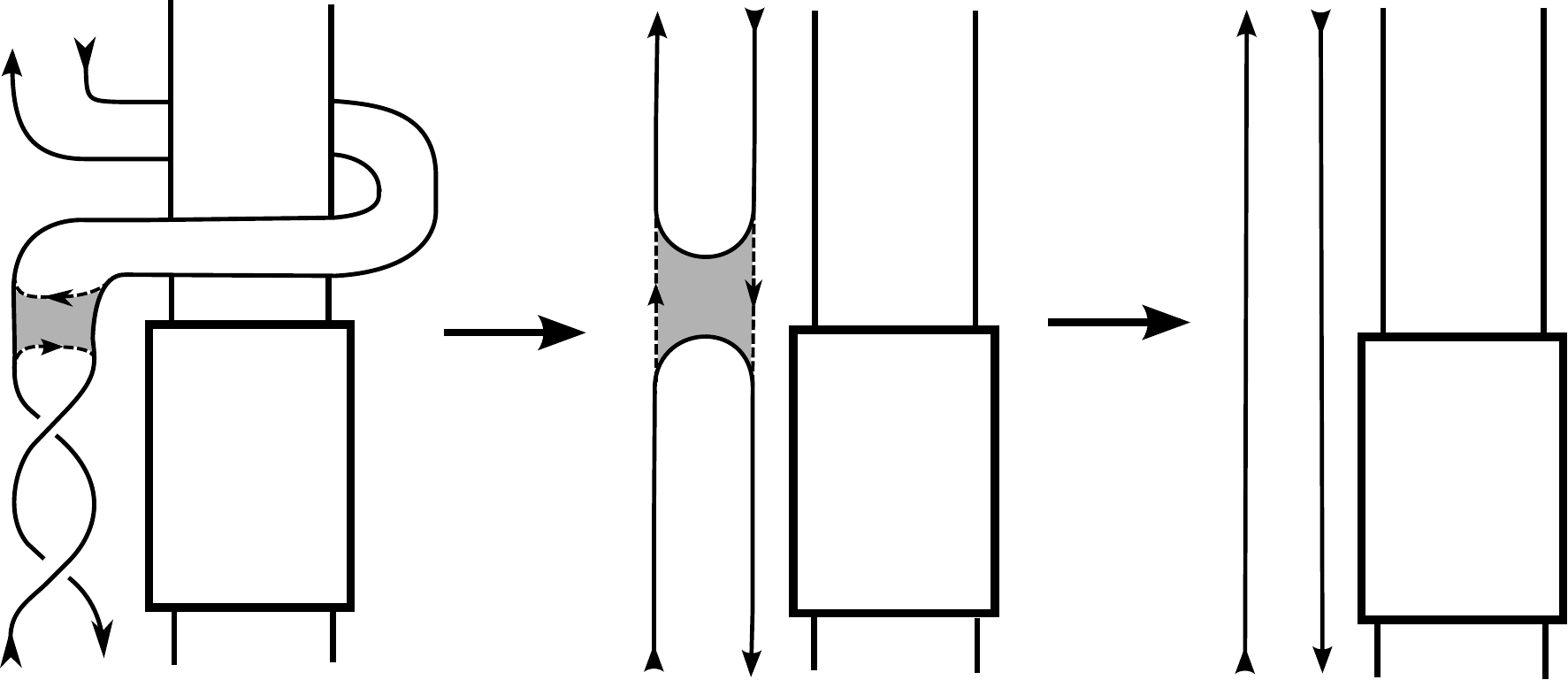}
    \put (9,0) {$\underbrace{\hspace{2.1cm} }_{}$}
    \put (7,-3.5) {{\large $2p-2$ strands}}
  \put (14,11) {{\LARGE $+1$}}
  \put (55,11) {{\LARGE $+1$}}
  \put (92,11) {{\LARGE $+1$}}
  \end{overpic}
  \vspace{0.5cm}
  \caption{Two oriented band moves convert a full twist on $2p$ strands into a full twist on $2p-2$ strands with two parallel strands.}
  \label{fig:band_additions}
\end{figure}
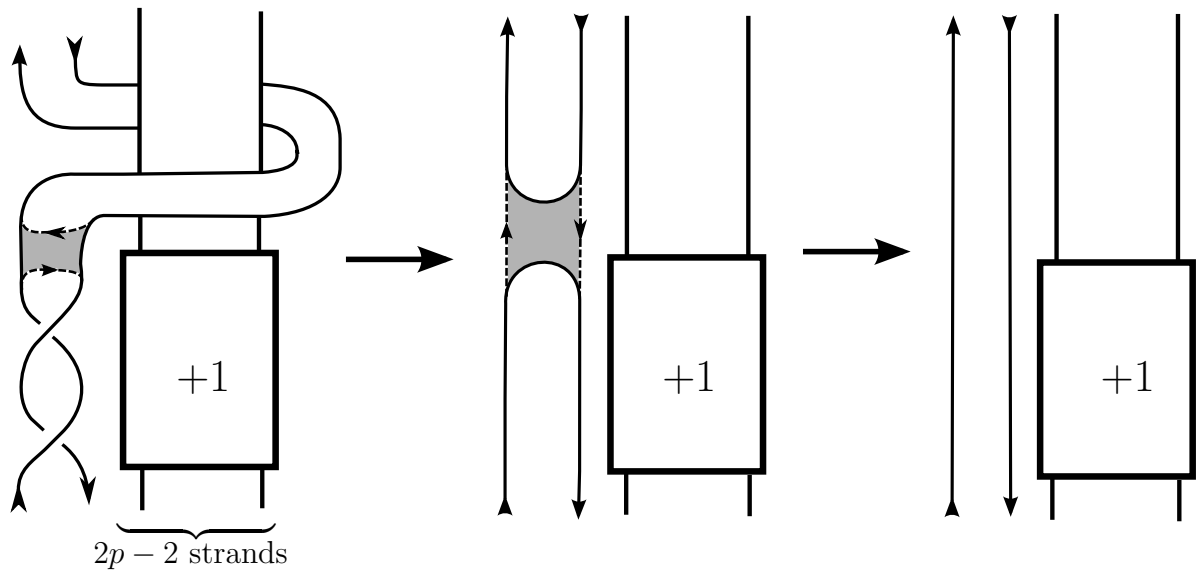
\end{proof}

Next we note how twisting operations transform under satellite operations. 
\begin{lem}\label{lem:satellite_move}
Let $K$ and $K'$ be knots related by a null-homologous twist on $2p$ strands. then for any pattern $P\subseteq S^1\times D^2$ with geometric winding number $w$, the satellites $P(K)$ and $P(K')$ are related by a null-homologous twist on $2pw$ strands.
\end{lem}
\begin{proof}
Let $X_P$ denote the complement $X_P=S^1\times D^2\setminus \nu P$ which comes with a distinguished meridian $\mu$ and $\lambda$ in $\partial (S^1\times D^2)$. The knot complement $S^3\setminus \nu P(K)$ is obtained by gluing $X_P$ to $S^3\setminus \nu K$ so that $\mu$ and $\lambda$ are glued to the meridian and null-homologous longitude of $K$ respectively. We can construct $S^3\setminus \nu P(K')$ similarly by gluing $X_P$ to $S^3\setminus \nu K'$.

Since $K$ and $K'$ are related by a null-homologous twist there is a null-homologous curve $C \subset S^3\setminus \nu K$ which can be surgered to obtain $S^3\setminus \nu K'$. Since $C$ is null-homologous in $S^3\setminus \nu K$, surgering $C$ takes the meridian and null-homologous longitude of $K$ to the meridian and null-homologous longitude of $K'$. We can consider $C$ as a curve in $S^3\setminus \nu P(K)= S^3\setminus \nu K \cup X_P$. Moreover surgering $C$ will produce $S^3\setminus \nu P(K')=S^3\setminus \nu K' \cup X_P$. 
Since $C$ is null-homologous in $S^3\setminus \nu K$ it is null-homologous in $S^3\setminus \nu P(K)$. Moreover if $C$ bounds a disk in  $S^3$ intersecting $K$ in $2p$ points and the $P$ has geometric winding number $w$, then $C$ bounds a disk intersecting $P(K)$ in $2pw$ points. Thus $P(K)$ and $P(K')$ are related by a null-homologous twist on $2pw$ strands, as required. 
\end{proof}

It immediately follows from Lemma~\ref{lem:satellite_move} that given a pattern $P$ with geometric winding number $w$ we have the following inequality:
\begin{equation}\label{eq:pw_upperbound}
tu_{pw}(P(K))\leq tu_p(K) + tu_{pw}(P(U)).
\end{equation}
Notice that Lemma~\ref{lem:satellite_move} also implies the following result.
\begin{corol}\label{cor:satellite}
For any knot $K$ and pattern $P$, we have
\[
tu(P(K))\leq tu(K) + tu(P(U)).
\]
\end{corol}
Although we won't use Corollary~\ref{cor:satellite} at any point in this paper,  we include it for comparison with an analogous inequality that exists for the algebraic genus \cite{Feller19satellite, McCoy19algebraic}.

Now we construct our examples. We will use Ozsv\'{a}th and Szab\'{o}'s $\tau$-invariant \cite{Ozsvath03Tau} to obtain lower bounds on $g_4(K)$. 

\unboundedgaps*
\begin{proof}
Set $n=m(p-1)$ and take $K$ to be any knot with $\tau(K)=u(K)=n$. For example, the torus knot $K=T_{2,2n+1}$. Let $K_p$ be the $(p,1)$-cable of $K$. Since the $(p,1)$-cable of the unknot is itself unknotted, it follows from \eqref{eq:pw_upperbound} that
\begin{equation}\label{eq:tu_p_upper}
tu_p(K_p)\leq tu_1(K)= n
\end{equation}
Now we compute the $\tau$-invariant of $K_p$ using the work of Hom \cite{Hom14Bordered}. The value of $\tau(K_p)$ depends on an auxiliary invariant $\varepsilon(K)$ which takes values in $\{-1,0,1\}$. Since $\tau(K)=u(K)$, we have that $\tau(K)=g_4(K)$. By \cite[Corollary~4]{Hom14Bordered} this implies that $\varepsilon(K)=1$. Thus the relevant formula for $\tau$ in \cite[Theorem~1]{Hom14Bordered} shows that
\[\tau(K_p)=p\tau(K)=pn.\]
Thus by \eqref{eq:tu_p_upper} and \eqref{eq:lowerbound}, we have
\[
\tau(K_p)=pn\leq g_4(K_p) \leq ptu_p(K_p)\leq pn.
\]
Hence $tu_p(K_p)=n$ and $g_4(K_p)=np$. So by applying \eqref{eq:lowerbound} to $tu_{p-1}(K_p)$ we have that
\[
tu_{p-1}(K_p)\geq \frac{np}{p-1}=n+\frac{n}{p-1}=n+m.
\]
Thus we have
\[
tu_{p-1}(K_p)-tu_{p}(K_p) \geq m,
\]
which is the required bound.
\end{proof}

\section{Untwisting torus knots}
Now we consider the untwisting numbers of torus knots.
\untwistingtorus*
\begin{proof}
\begin{figure}
  \begin{overpic}[width=0.95\textwidth]{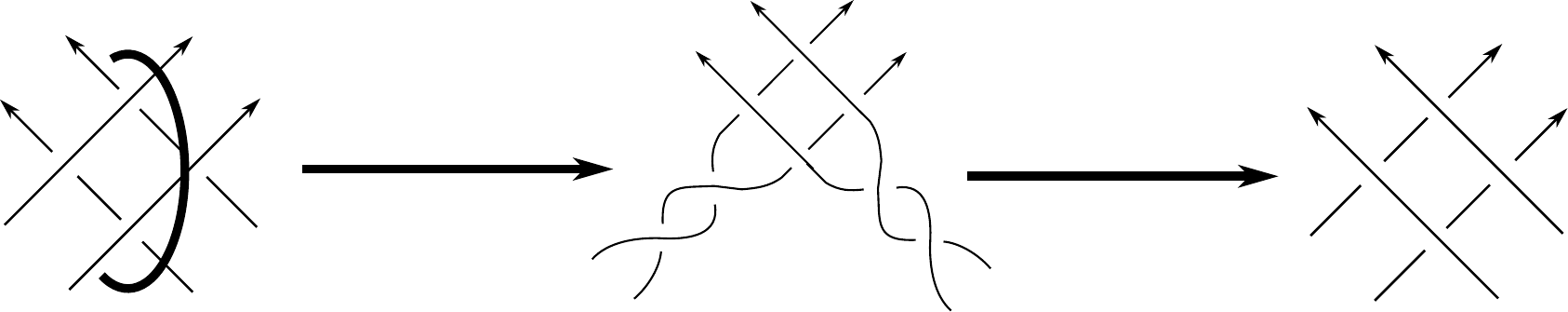}
    \put (20,11) {{\small twist on $4$ strands}}
    \put (61,11) {{\small two crossing changes}}
     \put (7,18) {$-1$}
  \end{overpic}
  \caption{Performing four crossing changes with three null-homologous twists.}
  \label{fig:double_change}
\end{figure}

Figure~\ref{fig:double_change} shows how a null-homologous twist on four strands followed by two crossing changes can be used to convert a square of four positive crossings into a square of negative crossings. We will refer to this operation as a `square change'.
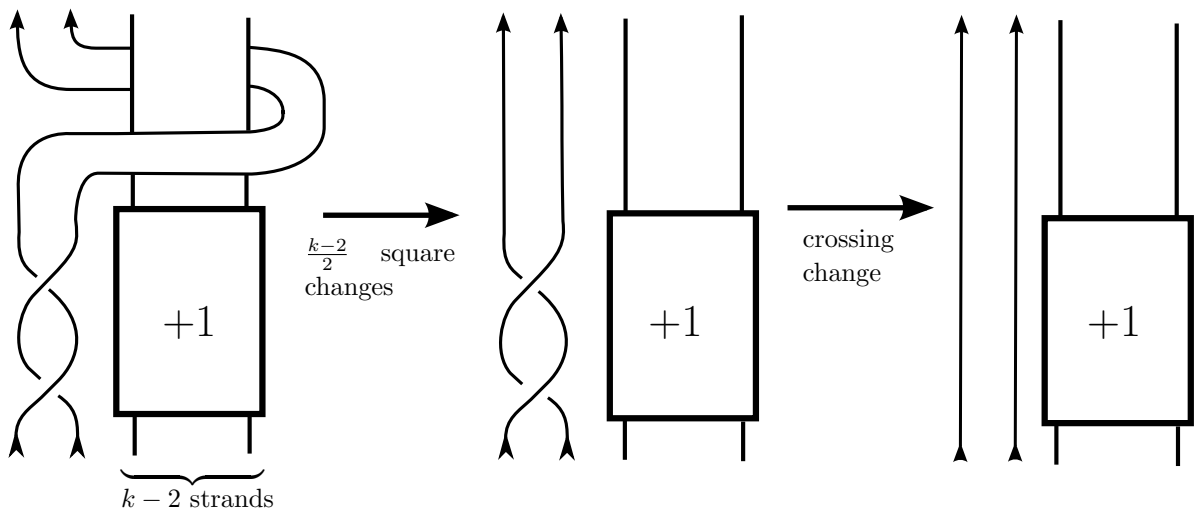
\begin{figure}
  \begin{overpic}[width=0.95\textwidth]{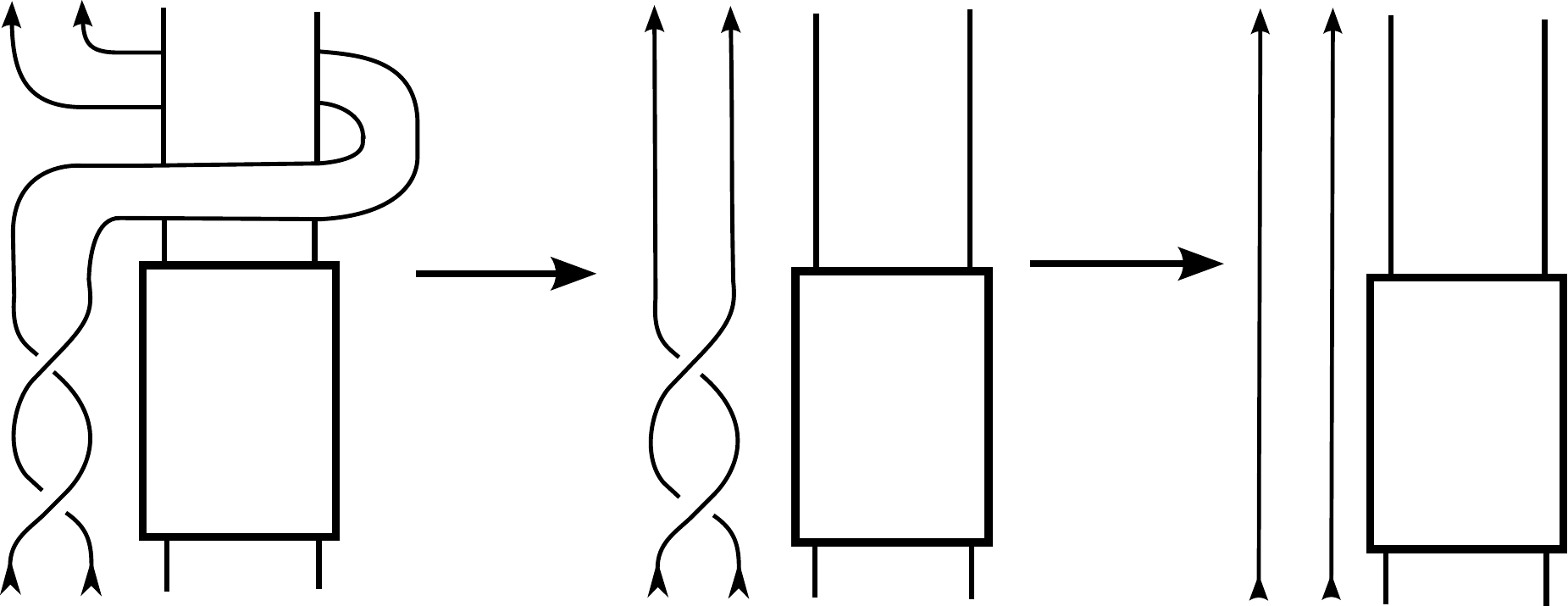}
    \put(9.5,-0.1){$\underbrace{\hspace{1.9cm} }_{}$}
    \put(9.5,-3.5){{\small $k-2$ strands}}
    \put (25,16) {\parbox{2cm}{{\small $\frac{k-2}{2}$ square changes}}}
    \put (67,17) {\parbox{2cm}{{\small crossing change}}}
    \put (91,11) {{\LARGE $+1$}}
    \put (54,11) {{\LARGE $+1$}}
    \put (13,11) {{\LARGE $+1$}}
  \end{overpic}
  \vspace{0.5cm}
  \caption{Using square changes to undo a full twist.}
  \label{fig:2moves}
\end{figure}
Suppose that we have a full twist on $k$ strands with the strands oriented in the same direction so that all the crossings are positive. We will assume first that $k$ is even. As shown in the right hand side of Figure~\ref{fig:2moves} we can view this full twist as a full twist on $k-2$ strands with two more strands wrapping round this full twist. As shown in Figure~\ref{fig:2moves}, we can convert this full twist into two parallel strands and a full twist on $k-2$ strands by taking the two strands and passing them through the other full twist on $k-2$ strands using $\frac{k-2}{2}$ square changes and performing a crossing change. This can achieved by $3\frac{k-2}{2} + 1 =\frac{3k}{2}-2$ null-homologous twists on at most four strands. Thus the full twist on $k$ strands can be converted into $k$ parallel strands by
\[
\sum_{i=1}^{\frac{k}{2}} \left(\frac{3k}{2}-2\right) = \frac{3k^2-2k}{8}
\]
null-homologous twists on at most four strands.

Now suppose that $k$ is odd. By performing $k-1$ crossing changes we can convert this to a full twist on $k-1$ strands with a single parallel strand alongside. The full twist on $k-1$ strands can then be undone as before, this shows that a full twist on $k$ strands can be converted into $k$ parallel strands by
\[
k-1+ \frac{3(k-1)^2-2(k-1)}{8}=\frac{3k^2-3}{8}
\]
null-homologous twists on at most four strands. Thus we see that for torus knots $tu_2$ satisfies the recursive upper bound
\begin{equation}\label{eq:tu2_recursion}
tu_2(T_{p, q})\leq  tu_2(T_{p-q,q})+ \begin{cases}
\frac{(3q^2-2q)}{8}  &\text{$q$ even}\\
\frac{(3q^2-3)}{8}  &\text{$q$ odd},
\end{cases}
\end{equation}
where $p>q\geq 2$. For comparison the unknotting number satisfies the recursion
\[
u(T_{p, q})=\frac{q(q-1)}{2}+ u(T_{p-q,q}).
\]
Thus we see that
\begin{align*}
u(T_{p,q})-tu_{2}(T_{p,q})&\geq u(T_{p-q,q})- tu_{2}(T_{p-q,q}) +\frac{q(q-1)}{2} - \begin{cases}
\frac{(3q^2-2q)}{8}  &\text{$q$ even}\\
\frac{(3q^2-3)}{8}  &\text{$q$ odd}
\end{cases}\\
& \geq \begin{cases}
\frac{(q^2-2q)}{8}  &\text{$q$ even}\\
\frac{(q^2-4q+3)}{8}  &\text{$q$ odd}
\end{cases}\\
&=\frac{1}{2}\qfloor \left(\qfloor -1\right).
\end{align*}
Since this final line is at least one whenever $q\geq 4$, this shows that $u(T_{p, q})>tu_2(T_{p, q})$ whenever $\min\{p,q\}\geq 4$.

Now we prove the upper bound \eqref{eq:tu_2_upperbound} by induction on $\min\{p,q\}$. Since $\min\{p,q\}=1$ implies that $T_{p, q}$ is unknotted, \eqref{eq:tu_2_upperbound} is vacuously true. Without loss of generality suppose that $p>q>1$ and that we can write $p=nq+r$ where $1\leq r <q$ and $n\geq 1$. Suppose inductively that $tu_2(T_{q, r})\leq \frac{3qr}{8}$. By applying \eqref{eq:tu2_recursion} $n$ times  we see that
\begin{align*}
tu_2(T_{p, q})&<\frac{3nq^2}{8} + tu_2(T_{q,r})\\
&\leq \frac{3nq^2}{8} + \frac{3rq}{8}= \frac{3pq}{8},
\end{align*}
as required.
\end{proof}

\bibliography{twist_bib}
\bibliographystyle{alpha}

\end{document}